\documentclass[reqno,12pt]{amsart}

\usepackage{amsmath, epsfig, cite}
\usepackage{amssymb}
\usepackage{amsfonts}
\usepackage{latexsym}
\usepackage{amsthm}
\usepackage{graphicx}
\newtheorem{theorem}{Theorem}[section]

\newtheorem{lemma}[theorem]{Lemma}

\theoremstyle{remark}

\numberwithin{equation}{section}
\newcommand{\binq}[2]{\genfrac{[}{]}{0mm}{0}{#1}{#2}}

\newcommand{\Rmnum}[1]

\allowdisplaybreaks

\author{Yifan Chen}
\address{DEPARTMENT OF MATHEMATICS, SHANGHAI UNIVERSITY, SHANGHAI 200444, P. R. CHINA}
\email{$^*$ Corresponding author. xiaoxiawang@shu.edu.cn (X. Wang), xchangi@shu.edu.cn (C. Xu),chenyf576@shu.edu.cn (Y. Chen).}

\author{Chang Xu}

\author{Xiaoxia Wang$^*$}

\title{Some new results about $q$-trinomial coefficients}

\begin{document}
\begin{abstract}
 In this paper, we present several new congruences on the $q$-trinomial coefficients
introduced by Andrews and Baxter. A new congruence on sums of
central $q$-binomial coefficients is also established.
\end{abstract}
\maketitle

\section{Introduction}
In 2019, Straub \cite{Straub2} gave the following $q$-supercongruence:
\begin{equation}\label{Eq3}
\binq{an}{bn}\equiv\binq{a}{b}_{q^{n^2}}-(a-b)b\dbinom{a}{b}\frac{n^2-1}{24}(q^n-1)^2\pmod{\Phi_n(q)^3},
\end{equation}
which is a $q$-analogue of the Wolstenholme--Ljunggren congruence: for any prime $p\geq5$,
\begin{equation*}
\dbinom{ap}{bp}\equiv\dbinom{a}{b}\pmod{p^3}.
\end{equation*}
Here and throughout the paper, $\Phi_n(q)$ stands for the $n$-th {\em cyclotomic polynomial} in $q$:
\begin{align*}
\Phi_n(q)=\prod_{\substack{1\leqslant k\leqslant n\\ \gcd(n,k)=1}}(q-\zeta^k),
\end{align*}
where $\zeta$ is an $n$-th primitive root of unity. The $q$-binomial coefficient is defined as
\begin{align*}
\binq{n}{k}=
\binq{n}{k}_q=
\begin{cases}
\displaystyle\frac{(q;q)_{n}}{(q;q)_{k}(q;q)_{n-k}}, &\text{if $0 \leq k \leq n$; }\\[4mm]
0, &\text{otherwise},
\end{cases}
\end{align*}
where the {\em $q$-shifted factorial} is defined as
$(a;q)_0=1$ and $(a;q)_n=(1-a)(1-aq)\cdots(1-aq^{n-1})$ with $n\in \mathbb{Z}^{+}$.

For $n\in \mathbb{N}$ and integer $j$ with $-n \leq j\leq n$, the trinomial coefficient is the coefficient of $x^j$ in the expansion of $(1+x+x^{-1})^n$. Namely,
\begin{equation*}
\bigg(\!\!\dbinom{n}{j}\!\!\bigg)= [x^j](1+x+x^{-1})^n,
\end{equation*}
and it has a simple expression (see\cite{Sills})
\begin{equation*}
\bigg(\!\!\dbinom{n}{j}\!\!\bigg)=\sum_{k=0}^{n}\dbinom{n}{k}\dbinom{n-k}{k+j}.
\end{equation*}
Six different $q$-analogues of the trinomial coefficients, which play significant roles in hard hexagon model, were introduced by Andrews and Baxter \cite{AndrewsBaxter}. One of these $q$-analogues is
\begin{equation}\label{Eq2}
\bigg(\!\!\dbinom{n}{j}\!\!\bigg)_q=\sum_{k=0}^{n}q^{k(k+j)}\binq{n}{k}\binq{n-k}{k+j}.
\end{equation}

During the past few years, some experts have paid attention to $q$-analogues of supercongruences.
We refer the reader to \cite{LiuWang1, LiuWang2, GuoZudilin, Guo4, Guo6} for some of their work.
Moreover, some congruences for $q$-binomial coefficients and $q$-trinomial coefficients can be found in \cite{Wang, NiPan, Zudilin, Liu4,Liu, Apagodu, GuoZeng}.

Recently, Liu \cite{Liu3} established the following beautiful $q$-supercongruences:
for any positive integer $n$, modulo $\Phi_n(q)^2$,
\begin{align}
\label{Eq25}\bigg(\!\!\dbinom{n}{0}\!\!\bigg)_q&\equiv \begin{cases}
 (-1)^m(1+q^m)q^{m(3m-1)/2},&\text{if \ $n=3m$};\\[4mm]
 (-1)^mq^{m(3m+1)/2},&\text{if \ $n=3m+1$};\\[4mm]
 (-1)^mq^{m(3m-1)/2},&\text{if \ $n=3m-1$},
\end{cases}\\[4mm]
\label{Eq4}\bigg(\!\!\dbinom{2n}{n}\!\!\bigg)_q&\equiv \begin{cases}
 2(-1)^m(1+q^m)q^{m(3m-1)/2}-3m(1-q^{3m}),&\text{if \ $n=3m$};\\[4mm]
 2(-1)^mq^{m(3m+1)/2}-(3m+1)(1-q^{3m+1}),&\text{if \ $n=3m+1$};\\[4mm]
 2(-1)^mq^{m(3m-1)/2}-(3m-1)(1-q^{3m-1}),&\text{if \ $n=3m-1$}.
\end{cases}
\end{align}

Motivated by the work just mentioned, we shall establish the following $q$-supercongruence similar to \eqref{Eq25} and \eqref{Eq4}.
\begin{theorem}\label{Thm8}
For any positive integer $n$, modulo ${\Phi_n(q)^2}$,
\begin{align}\label{Eq26}
\bigg(\!\!\dbinom{2n}{0}\!\!\bigg)_q
\equiv \begin{cases}
 2(-1)^m(1+q^m)q^{m(3m-1)/2}-9m(1-q^{3m})+1,&\text{if \ $n=3m$};\\[4mm]
 2(-1)^mq^{m(3m+1)/2}-3(3m+1)(1-q^{3m+1})+1,&\text{if \ $n=3m+1$};\\[4mm]
 2(-1)^mq^{m(3m-1)/2}-3(3m-1)(1-q^{3m-1})+1,&\text{if \ $n=3m-1$}.
\end{cases}
\end{align}
\end{theorem}
More generally, we shall give the following two new $q$-supercongruences.
\begin{theorem}\label{Thm2}
For any positive integers $a$ and $n$,  modulo ${\Phi_n(q)^2}$,
 \begin{align}\label{Eq7}
\bigg(\!\!\dbinom{an}{an-n}\!\!\bigg)_q\equiv \begin{cases}
 (-1)^m a(1+q^m)q^{m(3m-1)/2}-3m(1-q^{3m}){\binom{a}{2}},&\text{if \ $n=3m$};\\[5mm]
 (-1)^maq^{m(3m+1)/2}-(3m+1)(1-q^{3m+1})\binom{a}{2},&\text{if \ $n=3m+1$};\\[5mm]
 (-1)^maq^{m(3m-1)/2}-(3m-1)(1-q^{3m-1})\binom{a}{2},&\text{if \ $n=3m-1$}.
\end{cases}
\end{align}
\end{theorem}
\begin{theorem}\label{Thm3}
For any positive integers $a$ and $n$ with $a\geq 2$,  modulo ${\Phi_n(q)^2}$,
\begin{align}\label{Eq8}
&\bigg(\!\!\dbinom{an}{an-2n}\!\!\bigg)_q\nonumber\\[2mm]
&\equiv \begin{cases}
(-1)^m 2\binom{a}{2}(1+q^m)q^{m(3m-1)/2}-9m(1-q^{3m})\binom{a+1}{3}+\frac{3a-a^2}{2},&\text{if \ $n=3m$};\\[4mm]
(-1)^m 2\binom{a}{2}q^{m(3m+1)/2}-3(3m+1)(1-q^{3m+1})\binom{a+1}{3}+\frac{3a-a^2}{2},&\text{if \ $n=3m+1$};\\[4mm]
(-1)^m 2\binom{a}{2}q^{m(3m-1)/2}-3(3m-1)(1-q^{3m-1})\binom{a+1}{3}+\frac{3a-a^2}{2},&\text{if \ $n=3m-1$}.
\end{cases}
\end{align}
\end{theorem}
It is easily seen that Liu's two supercongruences \eqref{Eq25} and \eqref{Eq4} are the special cases of Theorem \ref{Thm2} with $a=1$ and $a=2$, respectively,
and Theorem \ref{Thm8} can be obtained by taking $a=2$ in Theorem \ref{Thm3}.

As serendipitous discoveries, we obtain the following $q$-congruences which are the contiguous forms of \eqref{Eq7} and \eqref{Eq8}  modulo $\Phi_n(q)$.
\begin{theorem}\label{Thm4}
For any integer $a \geq 1$, integer $j$ and positive odd integer $n$ with $0\leq j \leq n-1$, there holds
\begin{align}\label{Eq9}
\bigg(\!\!\dbinom{an-1}{an-n+j}\!\!\bigg)_q&\equiv\bigg(\frac{n-j}{3}\bigg)q^{(n-j-1)(n-j-2)/6-(j^{2}+j)/2}\pmod{\Phi_n(q)}.
\end{align}
Here and in what follows, $\big(\frac{.}{p}\big)$ denotes the Legendre symbol.
\end{theorem}
\begin{theorem}\label{Thm5}
For any integer $a \geq 2$, integer $j$ and positive odd integer $n$ with $0\leq j \leq n$, there holds, modulo $\Phi_n(q)$,
\begin{align}\label{Eq10}
\bigg(\!\!\dbinom{an-1}{an-2n+j}\!\!\bigg)_q&\equiv(a-1)\bigg(\frac{n-j}{3}\bigg)q^{\frac{(n-j-1)(n-j-2)}{6}-\frac{j^{2}+j}{2}}+\bigg(\frac{j}{3}\bigg)q^{\frac{1-j^{2}}{3}}.
\end{align}
\end{theorem}
It should be mentioned that if we take $j=0$ in Theorem \ref{Thm4} and $j=n$ in Theorem \ref{Thm5}, then the right-hand sides of \eqref{Eq9} and \eqref{Eq10}
are congruent to each other modulo $\Phi_n(q)$.
\begin{theorem}\label{Thm6}
For any integer $a \geq 2$, positive integer $b$ satisfying $a\leq b+2$ and positive odd integer $n$, there holds
\begin{align}\label{Eq11}
\bigg(\!\!\dbinom{an-1}{bn-1}\!\!\bigg)_q\equiv \dbinom{a}{b}+\dbinom{a-1}{b}\bigg(\frac{n+1}{3}\bigg)q^{n(n-1)/6}\pmod{\Phi_n(q)}.
\end{align}
\end{theorem}
By taking $j=n-1$ in Theorem \ref{Thm5} and $b=a-1$ in Theorem \ref{Thm6}, we can easily find that they have the same left-hand sides.
As a result, the right-hand sides of \eqref{Eq10} and \eqref{Eq11} are congruent to each other modulo $\Phi_n(q)$. In other words, for any integer $a \geq 2$ and positive odd integer $n$, there holds
\begin{equation*}
a+\bigg(\frac{n+1}{3}\bigg)q^{n(n-1)/6}\equiv (a-1)q^{-n(n-1)/2}+\bigg(\frac{n-1}{3}\bigg)q^{-n(n-2)/3}\pmod{\Phi_n(q)}.
\end{equation*}

The rest of the paper is arranged as follows. In the next section, we shall give proofs of Theorems \ref{Thm2} and \ref{Thm3}.
The proofs of Theorems \ref{Thm4} and \ref{Thm5} will be presented in Sections \ref{Sec4} and \ref{Sec5}, respectively. In the last section, a sketch of the proof of Theorem \ref{Thm6} will be provided.

\section{Proofs of Theorems \ref{Thm2} and \ref{Thm3}}\label{Sec2}
In order to prove Theorems \ref{Thm2} and \ref{Thm3}, we need the following two lemmas which have been proved by Liu.
\begin{lemma}[Liu\cite{Liu1}]\label{Lem1}
For any non-negative integer $n$, there holds
\begin{align}\label{Eq12}
(1-q^n)\sum_{k=0}^{\lfloor\frac{n}{2}\rfloor}\frac{(-1)^{k}q^{k(k-1)/2}}{1-q^{n-k}}\binq{n-k}{k}
=\begin{cases}
 (-1)^m(1+q^m)q^{m(3m-1)/2},&\text{if \ $n=3m$};\\[2mm]
 (-1)^mq^{m(3m+1)/2},&\text{if \ $n=3m+1$};\\[2mm]
 (-1)^mq^{m(3m-1)/2},&\text{if \ $n=3m-1$},
\end{cases}
\end{align}
where $\lfloor x\rfloor$ is the integral part of real $x$.
\end{lemma}
\begin{lemma}[Liu\cite{Liu3}]\label{Prop1}
For any positive integer $n$, there holds
\begin{equation}
\sum_{k=1}^{\lfloor\frac{n}{2}\rfloor}\frac{q^{-k(k-1)}}{[2k]_q}\binq{2k}{k}\equiv\frac{(1-q)(1-R_n(q))}{1-q^n}\pmod{\Phi_n(q)}\nonumber,
\end{equation}
where $R_n(q)$ denotes the right-hand side of \eqref{Eq12}.
\end{lemma}
It is not difficult to see that $R_n(q)\equiv 1\pmod{\Phi_n(q)}$.

\begin{proof}[Proof of Theorem \ref{Thm2}]
Firstly, we express the left-hand side of \eqref{Eq7} by Andrews and Baxter's expression \eqref{Eq2}. Note that${an-k\brack k+an-n}$ for any $k>\lfloor\frac{n}{2}\rfloor$. Therefore, we get
\begin{align}\label{Eq13}
\bigg(\!\!\binom{an}{an-n}\!\!\bigg)_q&=\sum_{k=0}^{an}q^{k(k+an-n)}\binq{an}{k}\binq{an-k}{k+an-n}\nonumber\\
&=\sum_{k=0}^{\lfloor\frac{n}{2}\rfloor}q^{k(k+an-n)}\binq{an}{k}\binq{an-k}{k+an-n}\nonumber\\
&=\binq{an}{an-n}+\sum_{k=1}^{\lfloor\frac{n}{2}\rfloor}q^{k(k+an-n)}\binq{an}{k}\binq{an-k}{k+an-n}.
\end{align}
For $1\leq k\leq \lfloor\frac{n}{2}\rfloor$, we have
\begin{align}\label{Eq14}
\binq{an}{k}&=\frac{(1-q^{an})(1-q^{an-1})\cdots(1-q^{an+1-k})}{(1-q)(1-q^2)\cdots(1-q^k)}\nonumber\\[2mm]
&\equiv\frac{a(1-q^n)(1-q^{-1})\cdots(1-q^{1-k})}{(1-q)(1-q^2)\cdots(1-q^k)}\nonumber\\[2mm]
&\equiv\frac{a(1-q^n)(-1)^{k-1}q^{-k(k-1)/2}}{1-q^{k}}\quad\pmod{\Phi_n(q)^2},
\end{align}
where the second relation is due to the fact that $1-q^{an}\equiv a(1-q^n)\pmod{\Phi_n(q)^2}$.

Because the reduced form of ${an\brack k}$ contains the factor $\Phi_n(q)$, it is enough to consider the following $q$-binomial coefficient  modulo $\Phi_n(q)$.
For $1\leq k\leq \lfloor\frac{n}{2}\rfloor$, we have
\begin{align}\label{Eq15}
\binq{an-k}{k+an-n}\nonumber
&=\binq{an}{an-n}\frac{(1-q^n)(1-q^{n-1})\cdots(1-q^{n+1-2k})}{(1-q^{an})\cdots(1-q^{an-k+1})(1-q^{(a-1)n+1})\cdots(1-q^{(a-1)n+k})}\\[3mm]
&\equiv\binq{an}{n}\frac{1-q^n}{1-q^{an}}\frac{(1-q^{-1})\cdots(1-q^{1-2k})}{(1-q^{-1})\cdots(1-q^{-k+1})(1-q)\cdots(1-q^k)}\nonumber\\[3mm]
&\equiv\binq{an}{n}\frac{(-1)^k}{a}q^{-(3k-1)k/2}\binq{2k-1}{k}\pmod{\Phi_n(q)}.
\end{align}
Substituting the results \eqref{Eq14} and \eqref{Eq15} into the right-hand side of \eqref{Eq13}, we arrive at
\begin{equation}\label{Eq16}
\bigg(\!\!\binom{an}{an-n}\!\!\bigg)_q\equiv \binq{an}{n}-(1-q^n)\binq{an}{n}\sum_{k=1}^{\lfloor\frac{n}{2}\rfloor}\frac{q^{-k(k-1)}}{1-q^k}\binq{2k-1}{k}\pmod{\Phi_n(q)^2}.
\end{equation}
Then by Lemma \ref{Prop1}, we obtain
\begin{align}\label{Eq17}
\sum_{k=1}^{\lfloor\frac{n}{2}\rfloor}\frac{q^{-k(k-1)}}{1-q^k}\binq{2k-1}{k}
=\frac{1}{1-q}\sum_{k=1}^{\lfloor\frac{n}{2}\rfloor}\frac{q^{-k(k-1)}}{[2k]_q}\binq{2k}{k}
\equiv\frac{1-R_n(q)}{1-q^n}\pmod{\Phi_n(q)}.
\end{align}

On the other hand, by the weaker version of \eqref{Eq3}, we have
\begin{equation}\label{Eq18}
\begin{split}
\binq{an}{n}&\equiv \binq{a}{1}_{q^{n^2}}
\equiv 1+q^{n^2}+q^{2n^2}+\cdots+q^{(a-1)n^2}\\
&\equiv a-((1-q^{n^2})+(1-q^{2n^2})+\cdots+(1-q^{(a-1)n^2}))\\
&\equiv a-n(1-q^n)\frac{a(a-1)}{2}\pmod{\Phi_n(q)^2},
\end{split}
\end{equation}
where we have utilized the fact $q^n\equiv 1\pmod{\Phi_n(q)}$ and the identity
\begin{equation*}
1-q^{kn^2}=(1-q^n)(1+q^n+q^{2n}+\cdots+q^{(kn-1)n}).
\end{equation*}
Finally, substituting \eqref{Eq17} and \eqref{Eq18} into the right-hand side of \eqref{Eq16}, we are led to
\begin{equation*}
\bigg(\!\!\binom{an}{an-n}\!\!\bigg)_q\equiv aR_n(q)-n(1-q^n)\dbinom{a}{2}\pmod{\Phi_n(q)^2}
\end{equation*}
as desired.
\end{proof}
\begin{proof}[Proof of Theorem \ref{Thm3}]
The proof is quite similar to that of Theorem \ref{Thm2}.
In order to apply Lemma \ref{Prop1}, we express \begin{tiny}$\big(\binom{an}{an-2n}\big)_q$\end{tiny} as follows:
\begin{equation*}
\bigg(\!\!\binom{an}{an-2n}\!\!\bigg)_q=\binq{an}{an-2n}+q^{(a-1)n^2}\binq{an}{n}+\sum_{k=1}^{n-1}q^{k(k+an-2n)}\binq{an}{k}\binq{an-k}{k+an-2n}.
\end{equation*}
For $1 \leq k \leq n-1$, we have
\begin{align}\label{Eq30}
\binq{an-k}{k+an-2n}\nonumber
&=\binq{an}{2n}\frac{(1-q^{2n})\cdots(1-q^{2n+1-2k})}{(1-q^{an})\cdots(1-q^{an-k+1})(1-q^{(a-2)n+1})\cdots(1-q^{(a-2)n+k})}\nonumber\\[3mm]
&\equiv\binq{an}{2n}\frac{1-q^{2n}}{1-q^{an}}\frac{(1-q^{-1})\cdots(1-q^{1-2k})}{(1-q^{-1})\cdots(1-q^{-k+1})(1-q)\cdots(1-q^k)}\nonumber\\[3mm]
&\equiv\binq{an}{2n}\:\frac{2}{a}\:(-1)^k q^{-(3k-1)k/2}\binq{2k-1}{k}\quad \pmod{\Phi_n(q)}.
\end{align}
Applying \eqref{Eq14}, \eqref{Eq17} and \eqref{Eq30}, we obtain
\begin{align}\label{Eq19}
\bigg(\!\!\binom{an}{an-2n}\!\!\bigg)_q&=\binq{an}{an-2n}+q^{(a-1)n^2}\binq{an}{n}+\sum_{k=1}^{n-1}q^{k(k+an-2n)}\binq{an}{k}\binq{an-k}{k+an-2n}\nonumber\\
&\equiv \binq{an}{2n}+q^{(a-1)n^2}\binq{an}{n}-2(1-q^n)\binq{an}{2n}\sum_{k=1}^{n-1}\frac{q^{-k(k-1)}}{1-q^k}\binq{2k-1}{k}\nonumber\\
&\equiv \binq{an}{2n}+q^{(a-1)n^2}\binq{an}{n}-2(1-q^n)\binq{an}{2n}\frac{1-R_n(q)}{1-q^n}\pmod{\Phi_n(q)^2}.
\end{align}
The third step holds since ${2k-1\brack k}\equiv 0 \pmod{\Phi_n(q)}$ for $\lfloor\frac{n}{2}\rfloor+1 \leq k \leq n-1$.

By utilizing the  method used in \eqref{Eq18}, we get
\begin{equation}\label{Eq20}
\begin{split}
\binq{an}{2n}\equiv \binq{a}{2}_{q^{n^2}}\equiv \frac{a(a-1)}{2}-n(1-q^n)\frac{a(a-1)(a-2)}{2}\pmod{\Phi_n(q)^2}.
\end{split}
\end{equation}
Substituting \eqref{Eq18} and \eqref{Eq20} into the right-hand side of \eqref{Eq19}, we arrive at
\begin{equation*}
\bigg(\!\!\binom{an}{an-2n}\!\!\bigg)_q\equiv -\frac{(a+1)a(a-1)}{2}n(1-q^n)+\frac{3a-a^2}{2}+a(a-1)R_n(q)\pmod{\Phi_n(q)^2}.
\end{equation*}
This completes the proof.
\end{proof}

\section{Proof of Theorem \ref{Thm4}}\label{Sec4}

For the sake of proving Theorem \ref{Thm4}, we need the following congruence on the sum of central $q$-binomial coefficients.
\begin{theorem}\label{Thm7}
For any positive odd integer $n$ and integer $j$ with $0\leq j\leq n-1$, there holds
\begin{equation*}
\sum_{k=0}^{n-j-1}q^{-k(k+1+j)}\binq{2k+j}{k}
\equiv (-1)^j\bigg(\frac{n-j}{3}\bigg)q^{(n-j-1)(n-j-2)/6} \pmod{\Phi_n(q)}.
\end{equation*}
\end{theorem}
Obviously, the special case of Theorem \ref{Thm7} with $j=0$ is
\begin{equation}\label{Eq22}
\sum_{k=0}^{n-1}q^{-k(k+1)}\binq{2k}{k}
\equiv\bigg(\frac{n}{3}\bigg)q^{(n-1)(n-2)/6} \pmod{\Phi_n(q)},
\end{equation}
which is a $q$-analogue of a congruence by Sun and Tauraso \cite{SunTauraso} (the modulo $p$ version): for any prime $p\geq5$,
\begin{equation*}
\sum_{k=0}^{p-1}\dbinom{2k}{k}\equiv\bigg(\frac{p}{3}\bigg)\pmod{p^2}.
\end{equation*}

In order to prove Theorem \ref{Thm7}, we recall the following result as the lemma, which firstly appeared in Ekhad and Zeilberger \cite{Ekhad} and Krattenthaler's note \cite{krat} and latter been proved by Warnaar \cite{warnaar} as a special case of a cubic summation formula.
\begin{lemma}\label{Lem2}
For any integer $n\geq0$, there holds
\begin{equation*}
\sum_{k=0}^{n}(-1)^kq^{k(k-1)/2}\binq{n-k}{k}
=(-1)^n\bigg(\frac{n+1}{3}\bigg)q^{n(n-1)/6}.
\end{equation*}
\end{lemma}
\begin{proof}[Proof of Theorem \ref{Thm7}]
Let $\omega=e^{2\pi mi/n}$ with $\gcd(m,n)=1$. Then we have that
 $\omega^n=1$ and $\omega^k\neq1$ for $1\leq k\leq n-1$.
Observing that
\begin{align*}
\omega^{-k(k+j+1)}\binq{2k+j}{k}_\omega=\omega^{-k(k+j+1)}\prod_{l=1}^{k} \frac{1-\omega^{2k+j+1-l}}{1-\omega^{l}}
=(-1)^k\omega^{k(k-1)/2}
\begin{bmatrix}n-k-j-1 \\k\end{bmatrix}_\omega,
\end{align*}
we immediately get
\begin{align*}
\sum_{k=0}^{n-j-1}q^{-k(k+j+1)}\binq{2k+j}{k}
&\equiv\sum_{k=0}^{n-j-1}(-1)^kq^{k(k-1)/2}\binq{n-k-j-1}{k}\\
&\equiv(-1)^j\bigg(\frac{n-j}{3}\bigg)q^{(n-j-2)(n-j-1)/6} \pmod{\Phi_n(q)},
\end{align*}
where we have utilized Lemma \ref{Lem2}.
\end{proof}
Now we begin to prove Theorem \ref{Thm4}.
\begin{proof}[Proof of Theorem \ref{Thm4}]
For $1\leq k \leq n-1$, we have
\begin{equation}\label{Eq23}
\binq{an-1}{k}=\frac{(1-q^{an-1})\cdots(1-q^{an-k})}{(1-q)\cdots(1-q^k)}
\equiv(-1)^kq^{-k(k+1)/2}\pmod{\Phi_n(q)}.
\end{equation}
We can also get the following congruence with $1 \leq k\leq \lfloor (n-1-j)/2\rfloor$,
\begin{align}\label{Eq24}
\binq{an-k-1}{an-n+k+j}
&=\binq{an}{n}\frac{(1-q^n)\cdots(1-q^{n-2k-j})}{(1-q^{an})\cdots(1-q^{an-k})(1-q^{(a-1)n+1})\cdots(1-q^{(a-1)n+k+j})}\nonumber\\[4mm]
&\equiv(-1)^{k+j}q^{-(3k+j+1)(k+j)/2}\binq{2k+j}{k} \quad \pmod{\Phi_n(q)}.
\end{align}
By combining the congruence \eqref{Eq23} with \eqref{Eq24}, we arrive at the following result
\begin{align}
 \bigg(\!\!\dbinom{an-1}{an-n+j}\!\!\bigg)_q&=\sum_{k=0}^{\lfloor\frac{n-j-1}{2}\rfloor}q^{k(k+an-n+j)}\binq{an-1}{k}\binq{an-k-1}{k+an-n+j}\nonumber\\
&\equiv\binq{an-1}{an-n+j}-(-1)^jq^{-(j^2+j)/2}\sum_{k=1}^{\lfloor\frac{n-j-1}{2}\rfloor}q^{-k^2-kj-k}\binq{2k+j}{k}\nonumber\\
&\equiv \bigg(\frac{n-j}{3}\bigg)q^{(n-j-1)(n-j-2)/6-(j^2+j)/2}\pmod{\Phi_n(q)},\nonumber
 \end{align}
where we have applied Theorem \ref{Thm7} and ${2k+j\brack k} \equiv 0 \pmod{\Phi_n(q)}$ for $\lfloor\frac{n-j-1}{2}\rfloor+1 \leq k \leq n-j-1$ in the last step.
\end{proof}

\section{Proof of Theorem \ref{Thm5}}\label{Sec5}
\begin{proof}[Proof of Theorem \ref{Thm5}]
Because of the different range of $k$, the proof of Theorem \ref{Thm5} is a little different from that of Theorem \ref{Thm4}. We first split the summation of the following $q$-trinomial coefficient into three parts as
\begin{align}\label{Eq32}
\bigg(\!\!\dbinom{an-1}{an-2n+j}\!\!\bigg)_q=&\binq{an-1}{an-2n+j}+\sum_{k=1}^{n-j-1}q^{k(k+an-2n+j)}\binq{an-1}{k}\binq{an-k-1}{k+an-2n+j}\nonumber\\
&+\sum_{k=n-j}^{\lfloor\frac{2n-j-1}{2}\rfloor}q^{k(k+an-2n+j)}\binq{an-1}{k}\binq{an-k-1}{k+an-2n+j}.
\end{align}
The first and the second parts of the right-hand side of \eqref{Eq32} can be handled just as what we have done in the proof of Theorem \ref{Thm4}.
Thus, we can get the following result with no difficulty,
\begin{align}
&\binq{an-1}{an-2n+j}+\sum_{k=1}^{n-j-1}q^{k(k+an-2n+j)}\binq{an-1}{k}\binq{an-k-1}{k+an-2n+j}\nonumber\\
&\quad \equiv(a-1)\bigg(\frac{n-j}{3}\bigg)q^{(n-j-1)(n-j-2)/6-(j^2+j)/2}\pmod{\Phi_n(q)}.\label{Eq33}
\end{align}
Now we begin to calculate the third part of the right-hand side of \eqref{Eq32}. Since the third part will disappear if $j=0$,
we only need to consider the case where $j$ is in the range $1 \leq j \leq n$. Noticing that ${an-k-1\brack k+an-2n+j}$ can be simplified through the method used in \eqref{Eq24}, and recalling the result \eqref{Eq23}, we have
\begin{align}\label{Eq34}
&\sum_{k=n-j}^{\lfloor\frac{2n-j-1}{2}\rfloor}q^{k(k+an-2n+j)}\binq{an-1}{k}\binq{an-k-1}{k+an-2n+j}\nonumber\\
&\ \equiv(-1)^jq^{-(j^2+j)/2}\sum_{k=n-j}^{\lfloor\frac{2n-j-1}{2}\rfloor}q^{-k^2-kj-k}\binq{2k+j}{k}\pmod{\Phi_n(q)}.
\end{align}
Performing $k\to n-j+k$ in the summation of the right-hand side of \eqref{Eq34}, we have
\begin{align}\label{Eq35}
\sum_{k=n-j}^{\lfloor\frac{2n-j-1}{2}\rfloor}q^{-k^2-kj-k}\binq{2k+j}{k}
&=\sum_{k=0}^{\lfloor\frac{j-1}{2}\rfloor}q^{-(k+n-j)^2-(k+n-j)j-(k+n-j)}\binq{2(k+n-j)+j}{k+n-j}\nonumber\\
&\equiv\sum_{k=0}^{\lfloor\frac{j-1}{2}\rfloor}q^{-k^2+kj-k+j}\binq{n+2k-j}{k}\nonumber\\
&\equiv(-1)^{n-j}\bigg(\frac{j}{3}\bigg)q^{(j-1)(j-2)/6} \pmod{\Phi_n(q)},
\end{align}
where we have utilized Theorem \ref{Thm7} with $j\to n-j$ in the last step.

Finally, by combining the formulas from \eqref{Eq32} to \eqref{Eq35} together, we obtain
\begin{align}\label{Eq37}
\bigg(\!\!\dbinom{an-1}{an-2n+j}\!\!\bigg)_q&\equiv(a-1)\bigg(\frac{n-j}{3}\bigg)q^{\frac{(n-j-1)(n-j-2)}{6}-\frac{j^2+j}{2}}+\bigg(\frac{j}{3}\bigg)q^{\frac{1-j^{2}}{3}}\pmod{\Phi_n(q)}.
\end{align}
The right-hand side of \eqref{Eq35} equals $0$ if $j=0$, and so \eqref{Eq37} also holds when $j=0$.
Thus we finish the proof of Theorem \ref{Thm5}.
\end{proof}

\section{The Proof of Theorem \ref{Thm6}}\label{Sec6}
\begin{proof}[Sketch of Proof]
For $1 \leq k\leq \lfloor (a-b)n/2\rfloor$, we have
\begin{align*}
 \binq{an-k-1}{k+bn-1}\nonumber
&=\binq{an}{bn}\frac{(1-q^{(a-b)n})\cdots(1-q^{(a-b)n-2k+1})}{(1-q^{bn+1})\cdots(1-q^{bn+k-1})(1-q^{an-k})\cdots(1-q^{an})}\nonumber\\[2mm]
&\equiv(-1)^{k-1}\dbinom{a-1}{b}q^{-3k(k-1)/2}\binq{2k-1}{k} \pmod{\Phi_n(q)}.
\end{align*}
Similarly to the proof of Theorem \ref{Thm4}, we get
 \begin{align}
 \bigg(\!\!\dbinom{an-1}{bn-1}\!\!\bigg)_q&=\sum_{k=0}^{\lfloor\frac{(a-b)n}{2}\rfloor}q^{k(k+bn-1)}\binq{an-1}{k}\binq{an-k-1}{k+bn-1}\nonumber\\
&\equiv\binq{an-1}{bn-1}-\sum_{k=1}^{\lfloor\frac{(a-b)n}{2}\rfloor}q^{-k^2}\binq{2k-1}{k}\nonumber\\
&\equiv\binq{an-1}{bn-1}-\sum_{k=1}^{\lfloor\frac{(a-b)n}{2}\rfloor}(-1)^kq^{k(k-1)/2}\binq{n-k}{k}\nonumber\\
&\equiv \dbinom{a}{b}+\dbinom{a-1}{b}\bigg(\frac{n+1}{3}\bigg)q^{n(n-1)/6}\pmod{\Phi_n(q)},\nonumber
 \end{align}
where we have applied Lemma \ref{Lem2} in the last step.
\end{proof}

\end{document}